\documentclass{article}
\usepackage{graphicx,amsfonts,amsmath,amsthm,amssymb,url}

\newtheorem{theorem}{Theorem}[section]

\newtheorem{prop}[theorem]{Proposition}
\newtheorem{conjecture}{Conjecture}
\newtheorem{cor}[theorem]{Corollary}

\newtheorem{question}[theorem]{Question}

\allowdisplaybreaks

\title{Note on the spectra of Steiner distance hypermatrices}
\author{Joshua Cooper and Zhibin Du}
\date{\today}

\begin{document}

\maketitle

\begin{abstract}
    The Steiner distance of a set of vertices in a graph is the fewest number of edges in any connected subgraph containing those vertices.  The order-$k$ Steiner distance hypermatrix of a graph is the $n$-dimensional array indexed by vertices, whose entries are the Steiner distances of their corresponding indices.  In the case of $k=2$, this reduces to the classical distance matrix of a graph.  Graham and Pollak showed in 1971 that the determinant of the distance matrix of a tree only depends on its number $n$ of vertices.  Here, we show that the hyperdeterminant of the Steiner distance hypermatrix of a tree vanishes if and only if (a) $n \geq 3$ and $k$ is odd, (b) $n=1$, or (c) $n=2$ and $k \equiv 1 \pmod{6}$.  Two proofs are presented of the $n=2$ case -- the other situations were handled previously -- and we use the argument further to show that the distance spectral radius for $n=2$ is equal to $2^{k-1}-1$.  Some related open questions are also discussed.
\end{abstract}

\section{Introduction}

Distance matrices are a natural object of study in graph theory.  An influential paper of Graham and Pollak (\cite{GrPo71}) showed, among other things, the surprising result that determinants of distance matrices of trees only depend on the number $n$ of vertices: $(1-n)(-2)^{n-2}$.  This led to a tremendous amount of scholarship concerning these matrices' spectral properties.  Interested readers are directed to \cite{AoHa14} for a thorough history.

Steiner distance, a generalization of pairwise distance to any set of vertices, was introduced by \cite{ChOeTiZo89}; an extensive survey can be found in \cite{Ma17}. Let $V(G)$ be the vertex set of $G$. The Steiner distance $d(S)$ of a set $S \subseteq V(G)$ of vertices in a graph $G$ is the fewest number of edges in any connected subgraph of $G$ containing $S$.  A straightforward generalization of distance matrices is then the order-$k$ ``Steiner distance hypermatrix'' of a graph $G$ on $n$ vertices, the order-$k$ dimension-$n$ array $D_k(G)$ whose $(v_1,\ldots,v_k) \in V(G)^k$ entry is $d(\{v_1,\ldots,v_k\})$.

Since Qi (\cite{Qi05}) defined symmetric hyperdeterminants and provided tools for studying spectra, it is natural to ask for multilinear generalizations of the Graham-Pollak Tree Theorem.  In general, hyperdeterminants are computationally intensive to compute and conceptually difficult to study.  However, \cite{CoTa24a} showed that, for trees $T$ on $n\geq 3$ vertices, $\det(D_k(T))$ is zero when $k$ is odd; and \cite{CoTa24b} showed that it is nonzero for $k$ even.  It is trivial that $\det(D_k(T))$ is $0$ when $n = 1$, so -- as far as the vanishing of this quantity is concerned -- the only remaining case to resolve is $n=2$, which is the main result of this note.

\begin{theorem} \label{thm:mainthm}
The Steiner distance hyperdeterminant $\det(D_k(T))$ of a tree $T$ on $n$ vertices vanishes iff $n=1$, $k$ is odd and $n>2$, or $k \equiv 1\, (\bmod\,6)$ and $n=2$.
\end{theorem}

We present two short proofs: one which is an application of Qi's version of Sylvester's formula, and another which uses a description of the characteristic polynomial of the all-ones hypermatrix.  An interesting aspect of our formula for the hyperdeterminant is that, up to sign, it equals ``Wendt's determinant'' (OEIS A048954) with well-known connections to Fermat's Last Theorem.  We also use our computations to describe exactly the spectral radius of $D_k(T)$, and present a few open questions in this area.

\section{Proofs}

Let $D := D_k(K_2)$ denote the order-$k$ $2 \times \cdots \times 2$ Steiner distance hypermatrix of the single-edge graph $K_2$.  An order-$k$ hypermatrix $M$ with index set $S$ is ``symmetric'' if, for each permutation $\sigma$ of $[k]$, the $(i_1,\ldots,i_k) \in S^k$ entry of $M$ equals the $(i_{\sigma(1)},\ldots, i_{\sigma(k)})$ entry.  Of course, if $S = \{0,1\}$, then the entries of a symmetric hypermatrix depend only on the number of their indices which equal $0$ or $1$.  In this special case, Qi gives the following version of Sylvester's formula in \cite{Qi05} for the hyperdeterminant:

\begin{prop}
    For an order-$k$, dimension-$2$ symmetric hypermatrix $M$ with entry $a_t$ at indices $t$ of whose coordinates are $1$, $0 \leq t \leq k$, $\det(M)$ is given by
    $$
    \left |
\begin{array}{ccccccccc}
a_0 & \binom{k-1}{1} a_1 & \binom{k-1}{2} a_2 & \cdots & a_{k-1} & 0 & \cdots & 0 & 0 \\
0 & a_0 & \binom{k-1}{1} a_1 & \cdots &\binom{k-1}{k-2} a_{k-2} & a_{k-1}  & \cdots & 0 & 0 \\
\vdots &&&& \vdots &&&& \vdots \\
0 & 0 & 0 & \cdots & \binom{k-1}{1} a_1 & \binom{k-1}{2} a_2 & \cdots & \binom{k-1}{k-2} a_{k-2} & a_{k-1} \\
%%%
a_1 & \binom{k-1}{1} a_2 & \binom{k-1}{2} a_3 & \cdots & a_{k} & 0 & \cdots & 0 & 0 \\
0 & a_1 & \binom{k-1}{1} a_2 & \cdots &\binom{k-1}{k-2} a_{k-1} & a_{k}  & \cdots & 0 & 0 \\
\vdots &&&& \vdots &&&& \vdots \\
0 & 0 & 0 & \cdots & \binom{k-1}{1} a_2 & \binom{k-1}{2} a_3 & \cdots & \binom{k-1}{k-2} a_{k-1} & a_{k}
\end{array}
\right |
    $$
\end{prop}
For the hypermatrix $D$, we have $a_0 = 0$, $a_k = 0$, $a_{t} = 1$ for all $0 < t < k$.  Therefore,
$$
\det(D) = \left |
\begin{array}{cccccccccc}
0 & \binom{k-1}{1} & \binom{k-1}{2} & \cdots & \binom{k-1}{k-2} & 1 & 0 & \cdots & 0 & 0 \\
0 & 0 & \binom{k-1}{1} & \cdots & \binom{k-1}{k-3} &\binom{k-1}{k-2} & 1 & \cdots & 0 & 0 \\
&&&& \vdots &&&&& \\
0 & 0 & 0 & \cdots & \binom{k-1}{1} & \binom{k-1}{2} & \cdots & \binom{k-1}{k-2} & 1 & 0 \\
0 & 0 & 0 & \cdots & 0 & \binom{k-1}{1} & \cdots & \binom{k-1}{k-3} & \binom{k-1}{k-2} & 1 \\
%%%
1 & \binom{k-1}{1} & \binom{k-1}{2} & \cdots & \binom{k-1}{k-2} & 0 & 0 & \cdots & 0 & 0 \\
0 & 1 & \binom{k-1}{1} & \cdots & \binom{k-1}{k-3} &\binom{k-1}{k-2} & 0 & \cdots & 0 & 0 \\
&&&& \vdots &&&&& \\
0 & 0 & 0 & \cdots & \binom{k-1}{1} & \binom{k-1}{2} & \cdots & \binom{k-1}{k-2} & 0 & 0 \\
0 & 0 & 0 & \cdots & 1 & \binom{k-1}{1} & \cdots & \binom{k-1}{k-3} & \binom{k-1}{k-2} & 0
\end{array}
\right | .
$$

Let
$$
A = \left [
\begin{array}{ccccc}
0 & \binom{k-1}{1} & \binom{k-1}{2} & \cdots & \binom{k-1}{k-2}  \\
0 & 0 & \binom{k-1}{1} & \cdots & \binom{k-1}{k-3}  \\
\vdots & \vdots & \vdots & \ddots & \vdots  \\
0 & 0 & 0 & \cdots & \binom{k-1}{1}
\end{array}
\right ]
$$
and
$$
B = \left [
\begin{array}{ccccc}
0 & 0 & \cdots & 0 & 0  \\
\binom{k-1}{k-2} & 0 & \cdots & 0 & 0 \\
\vdots & \vdots & \vdots & \ddots & \vdots  \\
\binom{k-1}{2} & \cdots & \binom{k-1}{k-2} & 0 & 0 \\
\binom{k-1}{1} & \cdots & \binom{k-1}{k-3} & \binom{k-1}{k-2} & 0
\end{array}
\right ].
$$
As usual, we use $I_k$ and $O_k$ to represent an identity matrix and a zero matrix of order $k$, respectively.

It is easy to see that $A + B + I_{k-1}$ is a circulant matrix, generated by the row $(\binom{k-1}{0},\binom{k-1}{1},\ldots,\binom{k-1}{k-2})$, thus the eigenvalues of $A + B + I_{k-1}$ are 
$$
\sum_{r = 0}^{k-2} \binom{k - 1}{r} \left( e^{\frac{2 \pi i j }{k-1}} \right)^r = \left ( 1 + e^{\frac{2 \pi i j}{k-1}} \right )^{k-1} - 1,
$$
where $j =0, \dots, k - 2$.  Next, note that
\begin{align*}
D = & \left [
\begin{array}{cc}
A & B + I_{k-1} \\
A + I_{k-1} & B
\end{array}
\right ] .
\end{align*}
Further, from
\begin{align*}
&\left [
\begin{array}{cc}
I_{k-1} & O_{k-1} \\
I_{k-1} & I_{k-1}
\end{array}
\right ]^{-1}
\left [
\begin{array}{cc}
A & B + I_{k-1} \\
A + I_{k - 1} & B
\end{array}
\right ]
\left [
\begin{array}{cc}
I_{k-1} & O_{k-1} \\
I_{k-1} & I_{k-1}
\end{array}
\right ]
\\
=&
\left [
\begin{array}{cc}
A + B + I_{k-1} & B + I_{k - 1}\\
O_{k-1} & - I_{k - 1}
\end{array}
\right ] ,
\end{align*}
we may conclude that the eigenvalues of $D$ are the eigenvalues of $A + B + I_{k-1}$, and $-1$ of multiplicity $k - 1$. Thus,
$$
\det (D) = (-1)^{k-1} \prod_{j=0}^{k-2} \left (\left ( 1 + e^{\frac{2 \pi i j}{k-1}} \right )^{k-1} - 1 \right ) .
$$
It is worth mentioning that
$$
(-1)^{k-1} \det (D) = \prod_{j=0}^{k-2} \left (\left ( 1 + e^{\frac{2 \pi i j}{k-1}} \right )^{k-1} - 1 \right )
$$
is Wendt's determinant $W_{k-1}$ (\cite[A048954]{oeis}); see \cite{He97} for example.  Lehmer showed (\cite{Le35}) that $W_m$ is $0$ iff $6 | m$, from which Theorem \ref{thm:mainthm} follows.

Now, we reprove this identity using characteristic polynomial $\det(\lambda I - D)$.  The following theorem appears in \cite{CoDu15}.  Denote by $J_n^k$ the all-ones, dimension-$n$, order-$k$ hypermatrix.

\begin{theorem} \label{thm:charpolyJ} Let \( \phi_n(\lambda) \) denote the characteristic polynomial of \( J_n^k \), \( n \geq 2 \).  Then
\[
\phi_n(\lambda) = \lambda^{(n-1) (k-1)^{n-1}} \prod_{\substack{r_1,\ldots,r_{k-1} \in \mathbb{N} \\ r_1 + \cdots + r_{k-1} = n}} \left ( \lambda - \left ( \sum_{j=1}^{k-1} r_j e^{\frac{2 \pi i j}{k-1}} \right )^{k-1} \right )^{\frac{\binom{n}{r_1,\ldots,r_{k-1}}}{k-1}}.
\]
\end{theorem}

If we take the $n=2$ case of this expression, the result is
\begin{align*}
\phi_n(\lambda) &= \lambda^{k-1} \prod_{\substack{r_1,\ldots,r_{k-1} \in \mathbb{N} \\ r_1 + \cdots + r_{k-1} = 2}} \left ( \lambda - \left ( \sum_{j=1}^{k-1} r_j e^{\frac{2 \pi i j}{k-1}} \right )^{k-1} \right )^{\frac{\binom{2}{r_1,\ldots,r_{k-1}}}{k-1}} \\
&= \lambda^{k-1}
\prod_{j_1, j_2 = 1}^{k-1} \left ( \lambda - \left ( e^{\frac{2 \pi i j_1}{k-1}} + e^{\frac{2 \pi i j_2}{k-1}} \right )^{k-1} \right )^{\frac{1}{k-1}}.
\end{align*}
Note that $D = J_2^k - I$, so the characteristic polynomial $\phi_D(\lambda)$ of $D$ is given by
\begin{align*}
\phi_D(\lambda) = \phi_n(\lambda + 1) &= (\lambda+1)^{k-1}
\prod_{j_1, j_2 = 1}^{k-1} \left ( \lambda - \left ( e^{\frac{2 \pi i j_1}{k-1}} + e^{\frac{2 \pi i j_2}{k-1}} \right )^{k-1} + 1 \right )^{\frac{1}{k-1}} \nonumber \\
&= (\lambda+1)^{k-1}
\prod_{j=0}^{k-2} \left ( \lambda - \left ( 1 + e^{\frac{2 \pi i j}{k-1}} \right )^{k-1} + 1 \right ). %\label{eq1}
\end{align*}
The constant term of this expression is
\begin{align*}
\det(D) &= \prod_{j=0}^{k-2} \left (- \left ( 1 + e^{\frac{2 \pi i j}{k-1}} \right )^{k-1} + 1 \right ) \\
&= (-1)^{k-1} \prod_{j=0}^{k-2} \left (\left ( 1 + e^{\frac{2 \pi i j}{k-1}} \right )^{k-1} - 1 \right ).
\end{align*}

We may now take advantage of the above analysis to compute the spectral radius of $D$.

\begin{cor}
    The Steiner distance hypermatrix of order $k$ with $n=2$ has spectral radius $2^{k-1}-1$.
\end{cor}
\begin{proof}
    Recall that the eigenvalues of $D$ are $-1$ with multiplicity $k-1$, and $[1+\exp(2 \pi i j/(k-1))]^{k-1}-1$ once for each $0 \leq j \leq k-2$. Since the point on the unit circle farthest from $z=-1$ is $1$, the maximum eigenvalue is achieved by taking $j=0$, i.e., $2^{k-1}-1$.
\end{proof}

\section{Conclusion}

We present a few open questions that arose in the present context.  First, we conjecture that the Graham-Pollak Tree Theorem has a full generalization to Steiner distance:

\begin{conjecture}
    The quantity $\det(D_k(T))$ is a function only of $n$ and $k$ for trees $T$ on $n$ vertices.
\end{conjecture}

The above is trivially true for $n \leq 3$ or $k$ odd, and \cite{CoTa24b} checked it computationally for $(k,n) = (4,4)$, $(4,5)$, and $(6,4)$.  We also venture the following conjecture, supported by all available evidence (as well as the Graham-Pollak results).

\begin{conjecture}
    Whenever it is nonzero, the sign of the quantity $\det(D_k(T))$ for trees $T$ on $n$ vertices is $(-1)^{n-1}$.
\end{conjecture}

Next, since \cite{FrGaHa13} showed that the Perron-Frobenius Theorem generalizes to hypermatrices like $D_k(G)$, and distance spectral radii have been studied extensively, we ask,

\begin{question}
    Provide bounds for the spectral radii of Steiner distance hypermatrices for trees and general graphs, in terms of their degree sequence and other statistics.
\end{question}

Finally, we wonder whether the extremal spectral radius is achieved by the path, as \cite{RuPo90} showed holds for ordinary distance matrices:

\begin{question}
    Is the largest spectral radius of the order-$k$ Steiner distance hypermatrix among all $n$-vertex graphs achieved by $P_n$?
\end{question}

\bibliographystyle{plain}
\bibliography{refs}

\end{document}